\documentclass{amsart}
\usepackage{cite}
\usepackage[all]{xy}
\usepackage{graphicx}
\usepackage{amssymb}
\usepackage{verbatim}
\usepackage[numbers,sort&compress]{natbib} 
\usepackage[bookmarksnumbered, bookmarksopen,
colorlinks,citecolor=blue,linkcolor=blue]{hyperref}

\newtheorem{Theorem}{Theorem}[section]
\newtheorem{Lemma}[Theorem]{Lemma}
\theoremstyle{Definition}
\newtheorem{Definition}[Theorem]{Definition}

\newtheorem{Proposition}[Theorem]{Proposition}

\theoremstyle{Remark}
\newtheorem{Remark}[Theorem]{Remark}

\theoremstyle{notation}

\numberwithin{equation}{section}

 \def\3{\operatorname{3}}

\begin{document}

\title{On the inverse limits of finite posets }

\author{Jing-Wen Gao   }
\address{School of Mathematics and Statistics, Huazhong University of Science and Technology, Wuhan 430074, P.R.China}

\email{d202280005@hust.edu.cn}

\author{Xiao-Song Yang  }

\thanks{} 

\subjclass[2020]{05E45, 06A07}

\date{\today}


\keywords{finite poset, barycentric subdivision, inverse limit.}

\begin{abstract}
In this paper, we show that any finite simplicial complex is homeomorphic to the inverse limit of a sequence of finite posets, which is an extension of Clader’s result.

\end{abstract}

\maketitle
\section{Introduction}
The theory of finite posets is based in a well-known correspondence between  topologies and preorders $\leq$ on a finite set $X$ given by Alexandroff \rm\cite{PSA1937}, and the topology on $X$ satisfies $T_0$ separation axiom if and only if the relation $\leq$ is a partial order.
 The homotopy theory of finite posets was introduced by
Stong \rm\cite{ST1966} and McCord \rm\cite{MC1966} from totally different perspectives. Stong \rm\cite{ST1966} focused on the combinatorics of finite posets to study their homotopy types. Stong showed that removing linear and colinear points (which are called up and beat points in \rm\cite{JP2003}) from a finite poset does not affect its homotopy type.
The main result McCord \rm\cite[Theorem 1]{MC1966} obtained in this paper is to establish a correspondence between finite simplicial complexes and finite posets up to weak homotopy equivalence. 

In recent years, there are many developments in the theory of finite posets.
For example, Barmak \rm{\cite {JAB2011}} studied some combinatorial problems by means of finite posets. N. Cianci and M. Ottina  applied spectral sequences to compute the homology groups of finite posets in \rm{\cite {NM2017}}, and described the Hurewicz cofibrations between finite posets in \rm{\cite {NM2019}}.

Our result in this paper makes progress toward an extension of Clader’s result.
 Clader {\rm\cite[Theorem 1.1]{EC}} obtained a generalization of McCord's result, showing that any finite simplicial complex is homotopy equivalent to an inverse limit of a sequence of finite posets. Expanding further upon Clader’s work, we prove a stronger version (see Theorem \ref{M2}) in Section \ref{4}, which states that any finite simplicial complex is homeomorphic to an inverse limit of a sequence of finite posets. As an application of Theorem \ref{M2}, we obtain an important approximation theorem (see Theorem \ref{M3}).

  In some situations, dealing with topological problems in the context of finite posets is more simple and tractable than in simplicial complexes. 
Theorem \ref{M3} provides a new viewpoint to approximate continuous maps in terms of morphisms in finite posets.

This paper is organized as follows. In Section \ref{2}, we give basic definitions and results concerning  finite posets and simplicial complexes. In Section \ref{4}, we associate a finite simplicial complex with an inverse system of finite posets and prove Theorem \ref{M2} and Theorem \ref{M3}.

\section{Preliminaries}\label{2}
In this section, we will review the fundamental facts relevant to this paper, refering the reader to  \rm{\cite{JAB2011, EC, SP1966}} for more details on finite posets and simplicial complexes.

An \emph{Alexandroff} space is a topological space in which any intersection of open subsets is still open. Given a set $X$, we endow it with a topology making it into an Alexandroff space is the same as impose a \emph{reflexive and transitive} relation (that is a \emph{preorder}) on it. This correspondence can be described briefly.
Let $X$ be an Alexandroff space, and for each $x\in X,$ let $U_x$ denote the intersection of all open sets that contain $x$. Then for $x, y\in X$, we define $x\leq y$ if $x\in U_y$. On the other hand, if $X$ has a reflexive and transitive relation, one can define a topology given by the basis $\{U_x~|~x\in X\}$, where $U_x=\{y\leq x~|~y\in X\}$.
It is easy to check that the \emph{antisymmetry} of the relation on $X$ (thus a \emph{partial order}) is equivalent to the $T_0$ separation axiom on the topology of $X$. A set with a partial order
is called a \emph{poset}. 

We review McCord's correspondence {\rm{\cite{MC1966}}} between finite posets and finite simplicial complexes. 
Given a finite poset $X$, $\mathcal{K}(X)$ is a simplicial complex whose vertices are the points of $X$, and simplices are the subsets of $X$ that are totally ordered.
Conversely, if $K$ is a finite simplicial complex, $\chi(K)$ denotes the face poset of $K$. That is, $\chi(K)$ is a poset of simplices of $K$ ordered by inclusion.

\begin{Theorem}{\rm (McCord)}
	\begin{itemize}
		\item [(1)] If $X$ is a finite poset, then there exists a weak homotopy equivalence $\mu_X: |\mathcal{K}(X)|\rightarrow X$.
		\item[(2)] If $K$ is a finite simplicial complex, then there exists a weak homotopy equivalence $\mu_K: |K|\rightarrow \chi(K)$.
	\end{itemize}
\end{Theorem}

A map $f: X\rightarrow Y$ between finite posets is \emph{order preserving} if $x\leq x^{\prime}$ implies $f(x)\leq f(x^{\prime})$ for every $x, x^{\prime}in X$.
\begin{Proposition}{\rm{\cite[Proposition 1.2.1] {JAB2011}}}\label{2.1}
	A map $f:X\rightarrow Y$ between finite posets is continuous if and only if it is order preserving.
\end{Proposition}

We now recall some notions about simplicial complexes.
Let $K=(V_K, S_K)$ be a simplicial complex, where $V_K$ is a set, called the set of vertices, and $S_K$ is a subset of the power set of $V_K$, called the set of simplices.
If $\sigma$ is a simplex of $K$, the \emph{star} $st(\sigma)$ is the subcomplex of $K$ whose simplices are the simplices $\tau$ of $K$ satisfying $\sigma\cup\tau\in K$.
The \emph{link} $lk(\sigma)$ is the subcomplex of $st(\sigma)$ whose simplices are disjoint from $\sigma$.

Let $\sigma=\{v_0, v_1,\cdots, v_n\}$ be a simplex of $K$. The \emph{closed} simplex $\overline{\sigma}$ is the set consisting of convex combinations $\sum\limits_{i=0}^n a_iv_i,~0\leq a_i\leq 1,~0\leq i\leq n,~\sum\limits_{i=0}^{n}a_i=1.$ The \emph{open} simplex $\mathring{\sigma}$ is the subset of $\overline{\sigma}$ consisting of $\sum\limits_{i=0}^n a_iv_i$ such that $a_i>0$ for every $0\leq i\leq n.$

A closed simplex $\overline{\sigma}$ is a 
metric space with the metric $d$ given by $$d(\sum\limits_{i=0}^na_iv_i, \sum\limits_{i=0}^nb_iv_i)
=\sqrt{\sum\limits_{i=0}^n(a_i-b_i)^2}.$$ The \emph{geometric realization} $|K|$ is the set that consists of convex combinations $\sum\limits_{v\in K}a_vv$ such that $\{v~|~a_v>0\}\in K$.
In other words, $|K|$ is the union of the closed simplices $\overline{\sigma}$
such that $\sigma\in K$.

The topology of $|K|$ is given as follows. A set $U\subset |K|$ is open (or closed) if and only if $U\cap \overline{\sigma}$ is open (or closed) in the metric space $\overline{\sigma}$ for every simplex $\sigma\in K$.

\section{An extension of  Clader's main theorem and its application }\label{4}
The techniques we use in the proof of our main result (Theorem \ref{M2}) are similar to those of Clader's {\rm\cite[Theorem 1.1]{EC}}. We keep the notation employed there. However, it should be noted that for each $n\geq 1$, the finite poset $X_n$ to be defined is different from the one in {\rm\cite{EC}}, and then  maps related to $X_n$ will be constructed in a different way.
   
We first recall the definition of  $n$-th barycentric subdivision.
\begin{Definition}
If $K$ is a finite simplicial complex, then $K_1=\mathcal{K}(\chi(K))$ is called the first barycentric subdivision of $K$. The second barycentric subdivision $K_2$ of $K$ is the barycentric subdivision of the first barycentric subdivision $K_1$, and so on inductively.
\end{Definition}

Let $K$ be a finite simplicial complex. For each $n\geq 0$, let $K_n$ denote the $n$-th barycentric subdivision of $K$. Corresponding to each $K_n$, let $X_n$ be a finite space whose points are the vertices of $K_n$.

 Before imposing a partial order on each $X_n, ~n\geq 1$, we consider a map $$p_n: |K|\rightarrow X_n,~ n\geq 1,$$ defined in the following way. Since every point $x$ of $|K|$ is contained in  precisely one open simplex $\mathring{\sigma}$ in $|K_{n-1}|$, and there is a unique vertex $b_n\in X_n$ lying in $\mathring{\sigma}$,  
we  define $p_n(x)=b_n.$ Therefore
$X_n$ can be made into a poset by means of $p_n$. For $x, y\in X_n,~x\leq y$ if and only if $\overline{p_n^{-1}(x)}\subset \overline{p_n^{-1}(y)}$.
The topology on the finite poset $X_n$ is given by the basis $\{B_x~|~x\in X_n\}$, where $B_x=\{y\in X_n~|y\geq x\}$.

We now show that $p_n$ is a weak homotopy equivalence.
It is not difficult to see that	$p_n$ is continuous, since for any $x\in X_n,$ we have
\begin{equation}\label{(4.1)}
p_n^{-1}(B_x)=\bigcup_{y\in B_x}\mathring{\sigma}_y^{n-1},
\end{equation}
where $\mathring{\sigma}_y^{n-1}$ denotes the open simplex in $|K_{n-1}|$ containing $y$.

\begin{Lemma}\label{4.1}
	For every $x\in X_n$, $B_x$ is contractible.
\end{Lemma}
\begin{proof}
	Define a map $r: B_x\rightarrow \{x\},~r(y)=x$. Then $r\geq id_{\{x\}}$, and from Proposition \ref{2.1}, $r\simeq id_{\{x\}}$.
\end{proof}

\begin{Lemma}\label{4.2}
	For every $x\in X_n, ~p_n^{-1}(B_x)$ is contractible.
\end{Lemma}
\begin{proof}
	Let $$\sigma_x^{n-1}=\bigcup_{z\in\overline{\sigma}_x^{n-1}} \text{support} (z),$$
	where $\overline{\sigma}_x^{n-1}$ is the closure of $\mathring{\sigma}_x^{n-1}$.
Then $p_n^{-1}(B_x)=\bigcup_{y\in B_x}\mathring{\sigma}_y=(st_{n-1}({\sigma}_x^{n-1})-lk_{n-1}({\sigma}_x^{n-1}))^{\circ}$, in which
 $st_{n-1}({\sigma}_x)$ and $lk_{n-1}({\sigma}_x^{n-1})$ denote the star and link of $\sigma_x^{n-1}$ in $K_{n-1}$ respectively. Because of the contractivity of $(st_{n-1}({\sigma}_x^{n-1})-lk_{n-1}({\sigma}_x^{n-1}))^{\circ}$, $~p_n^{-1}(B_x)$ is contractible.

\end{proof}

\begin{Remark}\label{Remark4}
{\rm	For $x=\sum\limits_{v\in K}a_vv\in |K|$, the support of $x$ is the simplex $$\text{\rm support}(x)=\{v~|~a_v>0\}.$$}
\end{Remark}

Before coming to our conclusion , we need the following Theorem.
\begin{Theorem}{\rm{\cite[Theorem 6]{MC1966}}}\label{4.3}
	Let $X$ and $Y$ be topological spaces, and let $f: X\rightarrow Y$ be any continuous map. Suppose that there exists an open cover $\mathcal{U}$ of $Y$ satisfying that for any $U, V\in \mathcal{U},$ and $x\in U\cap V,$ then there exists $W\in \mathcal{U}$ 
	such that $x\in W\subset U\cap V$. If for any $U\in \mathcal{U},$ the restriction
	$f|_{f^{-1}(U)}: f^{-1}(U)\rightarrow U$ is a weak homotopy equivalence, then $f: X\rightarrow Y$ is a weak homotopy equivalence.
\end{Theorem}
\begin{Theorem}
	$p_n$ is a weak homotopy equivalence.
\end{Theorem}
 \begin{proof}
It follows immediately from Lemma \ref{4.1}, Lemma \ref{4.2}, and Theorem \ref{4.3}.
\end{proof}

We refer the reader to \rm\cite{JJR2008} for basic facts on inverse limit. In this section, the inverse limit in the category of finite $T_0$ spaces will be considered.

There is a map  
$$q_n: X_n\rightarrow X_{n-1}, n\geq 2,$$ defined as follows.
For any $b_n\in X_n$, $b_n$ lies in exactly one open simplex $\mathring{\sigma}^{n-2}$ in $|K_{n-2}|$, there is a unique vertex $b_{n-1}\in X_{n-1}$ lying in $\mathring{\sigma}^{n-2}$,
let $q_n(b_n)=b_{n-1}$.
We then can define a family of maps $\{f_n^m: X_m\rightarrow X_n~ |~m\geq n\geq 1\}$ as follows.
For $m=n$, let $f_n^n=id.$ For $m>n$, let $f_n^m=q_m\circ\cdots\circ q_{n+1}$.
It is straightforward to check that maps $\{f_n^m\}$ satisfy $f_n^m=f^m_k\circ f_n^k,$ for $m\geq k\geq n$.
The continuity of $f_n^m$ follows from the following commutative diagram 
\begin{equation}\label{(4.2)}
\xymatrix{|K|\ar[d]_{p_m}\ar[dr]^{p_n} \\
	X_m \ar[r]^{f_n^m}  & X_n.\\
}
\end{equation}
Thus we have an inverse system of finite posets
$$\{X_n,~n\geq 1,~f_n^m,~m\geq n\geq 1\}.$$
Denote its inverse limit by $\widetilde{X}$.

\begin{Theorem}\label{M2}
	Let $K$ be a finite simplicial complex. Then $|K|$ is homeomorphic to the $\widetilde{X}$ defined above.
\end{Theorem}
\begin{proof}
	We retain the notations as above.
	Construct a map $F:\widetilde{X}\rightarrow |K|$ in the following way. Given $x=(x_1, x_2,\cdots)\in\widetilde{X},$ for each $x_n\in x$, we choose $a_n\in p_n^{-1}(x_n)=\mathring{\sigma}_{x_n}^{n-1}\subset|K_{n-1}| $. We thus obtain a sequence $\{a_n ~| ~n\geq 1\}$ of $|K|$. Observe that 
	$$\overline{\sigma}_{x_1}^{0}\supsetneqq\overline{\sigma}_{x_2}^{1}\supsetneqq\cdots\subsetneqq\overline{\sigma}_{x_n}^{n-1}\supsetneqq\cdots,$$
	and the maximum diameter of a closed simplex of $|K_{n}|$ approaches zero as $n$ approaches infinity, then $\{a_n ~| ~n\geq 1\}$ converges to a point $a\in |K|$,
	define $$ F(x)=a.$$
	If $\{b_n ~| ~n\geq 1\}$ is another sequence chosen in this way, we denote its limit by $\lim\limits_{n\to \infty}b_n=b$. For each real number $\epsilon>0$, there exists a positive integer $N$  such that for $n>N$, we have $|a_n-a|<\frac{\epsilon}{2}.$
	Note that there is a positive integer $N^{\prime}$ such that when $n>N^{\prime},~ |a_n-b_n|<\frac{\epsilon}{2}.$ Then $|b_n-a|<\epsilon$, if $n>N+N^{\prime}$, this implies $a=b$.
	Therefore $F$ is well-defined.
	
	We show that $F$ thus defined is injective.
	If $x\not=y\in \widetilde{X}$, then there exists a minimum $k$ such that $x_k\not=y_k$. Suppose that 
	$p_n^{-1}(x_n)=\mathring{\sigma}_{x_n}^{n-1},~ p_n^{-1}(y_n)=\mathring{\sigma}_{y_n}^{n-1}$, then when $n\geq k$, $$\mathring{\sigma}_{x_n}^{n-1}\cap\mathring{\sigma}_{y_n}^{n-1}=\emptyset. $$
	Let $\{a_n\}$ and  $\{b_n\}$ be sequences for the definitions of $F$ on $x$ and $y$ repectively.
In fact, for each $n$, there are open sets $U_n, V_n\subset |K|$ such that $a_n\in U_n\subset\overline{U}_n\subset \mathring{\sigma}_{x_n}^{n-1},~ b_n\in V_n\subset\overline{V}_n\subset \mathring{\sigma}_{y_n}^{n-1},$ it follows immediately that $\{a_n\}$ and  $\{b_n\}$ can not converge to a same point, which implies the injectivity of $F$.

Since for any $x\in |K|$, $F(p_1(x), p_2(x),\cdots)=x$, $F$ is surjective. Construct a map  $G: |K|\rightarrow \widetilde{X}$ as follows,
$$x\mapsto (p_1(x), p_2(x),\cdots).$$
Since the diagram $\ref{(4.2)}$ is commutative, $G$ is well defined. $G$ is obviously the inverse map of $F$.

We shall show that $G$ is an open map. It suffices to show that each $p_n$ is an open map.
Note that  a set $U\subset |K|$ is open if and only if $U\cap \overline{\sigma}$ is open in the metric space $\overline{\sigma}$ for every $\sigma\in K_{m}$ and for every  $m\geq 0$.
Let $U\subset |K|$ be an open set, if $z\in p_n(U)$, then there exists $x\in U$ such that $p_n(x)=z$. By definition of $p_n$, this implies that $x$ lies in the open simplex $\mathring{\sigma}_z^{n-1}\subset |K_{n-1}|$, and 
$$U\cap\mathring{\sigma}_z^{n-1}\not=\emptyset.$$
Then $U\cap\overline{\sigma}_z^{n-1}\not=\emptyset.$ If $z<y$, then $\overline{p_n^{-1}(z)}\subsetneqq \overline{p_n^{-1}(y)}$, that is $\overline{\sigma}_z^{n-1}\subsetneqq\overline{\sigma}_y^{n-1}$. Hence $U\cap\overline{\sigma}_y^{n-1}\not=\emptyset.$ Since $U$ is open in $|K|$, $U\cap\overline{\sigma}_y^{n-1}$ is open in the metric space $\overline{\sigma}_y^{n-1}$, then
$$U\cap\mathring{\sigma}_y^{n-1}\not=\emptyset.$$
This is equivalent to say that $y\in p_n(U)$. Hence we can obtain that 
$$B_z\subset p_n(U).$$
Therefore $p_n: |K|\rightarrow X_n$ is open. It follows that $G$ is an open map. Since $G$ is clearly continuous, $G$ is thus a homeomorphism.

\end{proof}
\begin{Remark}
	Theorem \ref{M2}  allows us to characterize continuous maps of polyhedra in terms of a morphism of inverse systems of finite posets.
\end{Remark}

Let $K, T$ be finite simplicial complexes. If $g: K\rightarrow T$ is a simplicial map,
it induces a continuous map $|g|: |K|\rightarrow |T|,$ and a  map $g_n: X_n\rightarrow Y_n$ for each $n\geq 1$.

 Consider the following diagram, 
 \begin{equation}\label{(4.3)}
 \xymatrix{&|K|\ar[r]^{p_n^{|K|}}\ar[d]_{|g|} & X_n\ar[d]^{g_n}\\
 	&|T|\ar[r]_{p_n^{|T|}} & Y_n.\\ 
 }
 \end{equation}
We now check that it is commutative. For $x\in |K|,$ let $\mathring{\sigma}\subset |K_{n-1}|$ be the open simplex containing $x$, and suppose that $b$ is the unique vertex of $X_n$ lying in $\mathring{\sigma}$, then $p_n^{|K|}(x)=b$. Since $g$ is a simplicial map, it maps the barycenter of any simplex $\tau$ of $K_{n-1}$ to the barycenter of simplex $g(\tau)$ of $T_{n-1}$. Thus $g_n(b)$ is the barycenter of $g(\sigma)$, in other words, $g_n(b)$ lies in the open simplex  $|g|(\mathring{\sigma})$, in which $|g|(x)$ is contained. Hence
 $$p_n^{|T|}\circ|g|(x)=g_n(b)=g_n\circ p_n^{|K|}(x).$$

\begin{Lemma}
$g_n: X_n\rightarrow Y_n$ is order preserving.
\end{Lemma}
\begin{proof}
For $a, b\in X_n$, if $a\leq b$, then $\overline{(p_n^{|K|})^{-1}(a)}\subset \overline{(p_n^{|K|})^{-1}(b)}$, that is $\overline{\sigma}_a^{n-1}\subset\overline{\sigma}_b^{n-1}$ in $|K_{n-1}|$. Since $g$ is a simplicial map, 
$$|g|(\overline{\sigma}_a^{n-1})\subset |g|(\overline{\sigma}_b^{n-1}).$$
Then
\begin{equation}\label{1}
|g|(\overline{(p_n^{|K|})^{-1}(a)})\subset |g|(\overline{(p_n^{|K|})^{-1}(b)}).
\end{equation}

In order to show that $g_n(a)\leq g_n(b)$, 
we need to show that $$\overline{(p_n^{|T|})^{-1} (g_n(a))}\subset \overline{(p_n^{|T|})^{-1} (g_n(b))},$$ but by the commutativity of the diagram \ref{(4.3)}, 
 this is precisely the formula \ref{1}.

\end{proof}
 
 It is not hard to see that $\{g_n\}: \{X_n\}\rightarrow \{Y_n\}$ is a morphism of inverse
 systems.
 Therefore we have the following commutative diagram,
 \begin{equation}
\xymatrix{&|K|\ar[r]^-{G^{|K|}}\ar[d]_{|g|} & \widetilde{X}={\rm inv~ lim}~X_n\ar[d]^{{\rm inv ~lim}~ g_n}_{\tilde{g}=}\\
 	&|T|\ar[r]_-{G^{|T|}}  & \widetilde{Y}={\rm inv~ lim}~Y_n.\\ 
 }
 \end{equation}
 
 \begin{Theorem}\label{M3}
 	Let $K, T$ be finite simplicial complexes. Any continuous map  $h: |K|\rightarrow |T|$ can be approximated by a continuous map of  inverse limits of a sequence of finite  posets. 
 \end{Theorem}
\begin{proof}
By the simplicial approximation theorem \rm\cite[p.76]{SW1975}, there is an integer $n\geq 1$ and a simplicial map $f: K_n\rightarrow T$ such that $|f|\simeq h$. Let $K^{\prime}=K_n$. The result follows immediately if we replace simplicial map $g: K\rightarrow T$ in the above argument by
simplicial map $f: K^{\prime}\rightarrow T$.

\end{proof}

\end{document}